\documentclass[leqno,11pt]{article}
\usepackage{float}
\usepackage{hyperref} 
\usepackage{amsmath, amssymb}
\usepackage{amsthm} 
\usepackage{amscd}
\usepackage{color} 
\usepackage[normalem]{ulem}
\usepackage{array}
\usepackage{verbatim} 
\usepackage{amsmath, amssymb}
\usepackage{amsthm} 
\usepackage{amscd}
\usepackage{multicol,lineno}%
\usepackage{enumitem}
\theoremstyle{plain} 
\newtheorem{theorem}{\indent\sc Theorem}[section]
\newtheorem{lemma}[theorem]{\indent\sc Lemma}
\newtheorem{corollary}[theorem]{\indent\sc Corollary}
\newtheorem{proposition}[theorem]{\indent\sc Proposition}
\theoremstyle{definition} 
\newtheorem{definition}{\indent\sc Definition}
\newtheorem{remark}[theorem]{\indent\sc Remark}

\newtheorem{notation}[theorem]{\indent\sc Notation}

\newcommand{\rexp}{\rm{exp}}
\newcommand{\rden}{\rm{den}}

\newcommand{\C}{\mathbb{C}}
\newcommand{\R}{\mathbb{R}}
\newcommand{\Q}{\mathbb{Q}}
\newcommand{\Z}{\mathbb{Z}}

\newcommand{\AND}{~~\textrm{and}~~}

\newcommand{\al}{\alpha}

\newcommand\ee{\varepsilon}
\newcommand\eps{\varepsilon}

\newcommand{\hh}{\mathrm{h}}

\newcommand{\mueff}{\mu_{\textrm{eff}}}


\begin{document}

\title{Approximation measures for shifted logarithms of algebraic numbers \\via effective Poincar\'e-Perron} 
\author{Noriko Hirata-Kohno, Ryuji Muroi and Yusuke Washio}
\date{}
\maketitle

\begin{abstract}

In this article,  we show new effective approximation measures
for the \textit{shifted} logarithm of algebraic numbers
which belong to a number field of  arbitrary degree. 
Our measures refine previous ones including those for usual logarithms.
We adapt Pad\'e approximants constructed by M.~Kawashima and A.~Po\"{e}ls in the rational case.
Our key ingredient relies on the Poincar\'e-Perron theorem which gives us 
asymptotic estimates at \textit{every} archimedean place of the number field,
that we apply to decide whether the shifted logarithm lies outside of the number field. 
We use  Perron's second theorem and its modification due to M.~Pituk
to handle general cases. 

\end{abstract}
{AMS Subject Classification and keywords\footnote{
{2020 \emph{AMS Subject Classification}: Primary 11J72; Secondary 11J82, 33C05,  33C20, 11B37, 65Q30.}
\\
{\emph{Key words}: Diophantine approximation, Pad\'e approximation,  irrationality measure, approximation measure,  shifted logarithm,
hypergeometric function, linear recurrence sequence, Poincar\'e-Perron theorem}}}

\section{Introduction}\label{Intro}

The Poincar\'e-Perron theorem in the theory of linear difference equations
has been adapted by K.~Alladi-M.~L.~Robinson \cite{A-R}, G.~Chudnovsky \cite{ch9} and M.~Kawashima-A.~Po\"{e}ls \cite{KP}, 
to obtain sharp irrationality measures. 
The theorem is applied to obtain asymptotic analytic estimates in arithmetic problems, especially in the so-called Pad\'e approximations \cite{Pade1,Pade2}
giving often a criterion for the irrationality of numbers in question.

The present article gives an effective generalization in the algebraic case
of the irrationality measures proven in \cite{KP}. In particular we show explicit examples of non-quadraticity measures 
as well as general approximation measures,
for the value at algebraic points of the shifted logarithmic function:~a specific hypergeometric function.
The construction of Pad\'e approximants in \cite{KP}  is endowed with a linear recurrence sequence,  that allows us to apply the Poincar\'e-Perron theorem which provides \textit{asymptotic} estimates. 
Although related linear independence criteria are shown by  Pad\'e approximations in \cite{DHK2, DHK3}, these linear independence measures
are not as sharp as ours because of the lack of asymptotic estimate. 

We follow the formulations due to F.~Amoroso-C.~Viola \cite{A-V} and R.~Marcovecchio-Viola  \cite{marco2} to obtain
effective approximation measures generalizing the irrationality ones, but we
replace the saddle point method used in \cite{A-V}  \cite{marco2},  by the Poincar\'e-Perron theorem.

Note that the transcendence of values of the shifted logarithmic function, more generally that of the Lerch function
at algebraic points, does not follow from Baker's theory \cite{Baker} except for the usual logarithms, because of the \textit{shift}.
Neither the transcendence criterion in \cite{bw} \cite{wo1} yields general approximation measures of our form.
We also remark that the shift of the Lerch function is related to an expression of the values of the Riemann zeta function \cite{Z}. 

\

Let $\xi\in \R\setminus \Q.$ The non-negative real number $\mu$ 
is an irrationality measure of $\xi$ if, for any $\eps>0$, there exists a constant $q_0=q_0(\xi, \varepsilon)>0$ such that
every rational number ${p/q}$
with $q \geq q_0$ satisfies 
\begin{align}\label{eq: exponent}
\vert \xi - {p/q} \vert> 1/{q^{\mu(1+ \eps)}}.
\end{align}
If the constant $q_0$ is effectively computable for all 
$\eps>0$, then $\mu$ is called an effective irrationality measure for $\xi$. 

For an algebraic number field $K$, let us define a $K$-approximation measure, 
as defined in \cite[Definition 2.1]{A-V} \cite[Definition 2.1]{marco2}.
We write $ K_{\infty} $ the completion of $K$ by an archimedean metric. 
Denote by $\hh$ the Weil absolute logarithmic height, defined in Section~\ref{section:height}.

\begin{definition}
Let $K\subset \C$ be an extension of $\Q$ of finite degree and $\xi\in K_{\infty} \setminus K$.
We say that a non-negative real number $\mu$ 
is a $K$-approximation measure of $\xi$, if for any $\varepsilon>0$ there exists a constant 
$h_0=h_0(\xi, \varepsilon)>0$ such that
\begin{align*}
    \log\left\vert\xi-\alpha\right\vert\geq -(1+\varepsilon)\cdot\mu\cdot \hh(\alpha)
\end{align*}
for any algebraic numbers $\alpha\in K$ with $\hh(\alpha) \geq h_0$.
If the constant $h_0$ is effectively computable for all $\ee>0$, then $\mu$ is called an effective $K$-approximation measure. 
The least $K$-approximation measure of $\xi$ is denoted by $\mu(\xi, K)$, and an effective one 
is denoted by $\mueff=\mueff(\xi, K)$, that is
\begin{align}\label{eq: exponent}
    \log\left\vert\xi-\alpha\right\vert\geq -(1+\varepsilon)\cdot\mueff\cdot \hh(\alpha)
\end{align}
for any algebraic numbers $\alpha\in K$  of height $\hh(\alpha)\geq h_0$ with an effective constant $h_0>0$.

\end{definition}
When $K=\Q$ and $\xi\in\R\setminus\Q$, the theory of continued fractions guarantees that $\mu\geq 2$, with $\mu(\xi, \Q)=2$ for almost all $\xi\in \R$ in the sense of the Lebesgue measure. The 
theorem of K.~F.~Roth implies $\mu(\xi, \Q)=2$ if $\xi$ is an irrational algebraic number, but
we do not know how to compute the constant $h_0>0$, thus Roth's theorem is not effective.
In this article, all the steps of our proof are effectively performed.

The articles \cite{A-V, Ha2, marco2}  give
$K$-approximation measures in the algebraic case by the saddle-point method, for the usual logarithms {without} shift.
For general polylogarithms or values of the Lerch function of algebraic numbers, E.~M.~Niki\v{s}in \cite{N},
A.~I.~Galochkin \cite{G2}, G.~V.~Chudnovsky \cite{ch11}, M.~Hata \cite{Ha} and K.~V$\ddot{\text{a}}$$\ddot{\text{a}}$n$\ddot{\text{a}}$nen  \cite{Va} gave 
linear independence criteria, mainly over the  rationals
or quadratic imaginary fields. Refer to \cite{DHK2, DHK3} a precise comparison with previous results and a historical survey.

The principal difference between results in \cite{A-V, DHK2, DHK3, Ha, Ha2, marco2} and ours, is  
as follows. First, 
our method relies on the Poincar\'e-Perron theorem (known in the archimedean case \textit{only}), 
as are done in the rational case in \cite{A-R, ch9, Hab, KP}. 
We use the Poincar\'e-Perron theorem $[K\,:\,\Q]$ times 
at each archimedean place, for our generalization to the algebraic case. 
Thanks to \textit{asymptotic} estimates in  Pad\'e approximation deduced from  the Poincar\'e-Perron theorem,  we obtain the result. 
Second, in Section~\ref{PP}, we apply Pituk's generalization of Perron's Second theorem.
Since we checked the effectivity of all the proof the theorem of Pituk \cite[Theorem 2]{Pituk} which generalizes Perron's Second theorem \cite{Perron21} (see also \cite{Perron09} \cite{Perron29}), we adapt 
effective statement due to Pituk \cite[Theorem 2]{Pituk} presented as Theorem \ref{Pituk} in Section~\ref{PP} in this article.

This article is organized as follows. 
In Section~\ref{section:height}, we review the definition of the Weil absolute logarithmic height.
The notations and our main result are stated in Section~\ref{notations}. 
Section~\ref{lemma} is devoted to introduce a lemma on
general approximation measures in \cite{A-V, marco2}.
In Section~\ref{Pade}, we recall the Pad\'e approximants 
constructed in  \cite[Proposition~3.5]{KP}, 
those satisfy a linear recurrence sequence \cite[Lemma 3.8]{KP}.
Pituk's generalization of the Poincar\'e-Perron theorem is given in Section~\ref{PP} and
we prove our main theorem in Section~\ref{proof}. 
Numerical explicit examples are shown in Section~\ref{example}.

\section{Heights}\label{section:height}
Throughout the article, for $z\in\C$, write $|z|$ its usual absolute value.
Let $K$ be a number field of finite degree and put $d=[K\,:\,\Q]$. 
Denote the set of places of $K$  by ${{{M}}}_K$,
respectively by ${{M}}^{\infty}_K$ for 
the archimedean places, and by ${{{M}}}^{f}_K$
for the non- archimedean places. We shortly write $v|\infty$ when $v\in{{M}}^{\infty}_K$ and $v\kern-2.3pt\not |\infty$ when $v\in{{{M}}}^{f}_K$.

\begin{definition}
For a place $v$ of $K$, we denote the normalized absolute value by $| \cdot |_v$ defined as follows:
\begin{align*}
&|p|_v=p^{-1} \quad~(v\kern-2.3pt\not |\infty, ~v|p)\enspace,\quad  |x|_v=|\sigma(x)|\quad(v|\infty)\enspace,
\end{align*}
where $p$ is the rational prime above $v$ when $v$ is finite, and 
$\sigma=\sigma_k~(1\leq k \leq d)$ is a conjugate map, namely one
embedding $K\hookrightarrow \C$ corresponding to $v$ infinite. Denote by
$d_v=[K_v:\Q_v]$ the local degree for $v$, with  $d_v=[K_{\infty}:\R]=1$ if $\sigma$ real,  and $d_v=[K_{\infty}:\R]=2$ otherwise. 
    For each $v\in {{M}_K}$ and $\al\in K$, we define the \textit{Weil absolute  logarithmic height} of  $\al$ by
 
\begin{center}
$\hh(\al)=\big(1/{d}\big)\sum_{v\in {{M}}_K} d_v\log^{+} |\al|_v$
\end{center}
with $\log^{+}x=\log\max\{x, 1\}$ for any $x\geq 0$.
Put also 
\begin{center}
$\hh_{\infty}(\al)=\big(1/d\big)\sum_{v\in{{M}}^{\infty}_K} d_v\log^{+} |\al|_v \AND 
\hh_{f}(\al)=\big(1/d\big)\sum_{v\in{{{M}}}^{f}_K} d_v \log^{+} |\al|_v.$
\end{center}
\end{definition}
 
\begin{definition}
Define $\delta=[K : \Q] / [K_{\infty} : \R]$ for $K_{\infty}=\R$ if $K\subset \R$ and $K_{\infty}=\C$ otherwise. 
Then 
\begin{equation}\label{delta}
\delta=d \text{~~when~~}  K  \text{~~is~real,~}  \delta=d/2  \text{~~otherwise}.
\end{equation}  
By the Liouville inequality~\cite[(2.\,5)]{A-V} for $\alpha\neq 0$, it follows $\log \vert\alpha\vert\geq -\delta \hh(\alpha)$, and
by the box principle~\cite[(2.\,6)]{A-V}\cite[p.\,253]{Schmidt}~, we have 
$\mu(\xi, K)\geq 2\delta$.

Let us consider an algebraic number $\beta$ under suitable hypothesis being specified later.
We shall restrict ourselves to $\xi=f(\beta)$ where $f(z)$ is a certain hypergeometric function defined in the next section.

\section{Notations and the main result}\label{notations}
Throughout the article, $n$ denotes a non-negative rational integer in $\Z_{\geq 0}$. 

\begin{definition}
Let  $z\in \C$ with $\vert z \vert> 1$ and $ x\in\mathbb{Q}\cap[0,1)$.
Consider $ \Phi_s(x,z) =\sum_{k=0}^{\infty}z^{k+1}/{(k+x+1)}^s$,  the $s$-th Lerch function.
Let us define \textit{the shifted logarithmic function} with a shift $x$ by
\begin{equation}\label{eq: three special cases for f}
f(z)=
(1+x)\Phi_1(x,1/z)={(1+x)}{\textstyle{\sum_{k=0}^{\infty}}}\,1/\Big((k+x+1)\cdot z^{k+1}\Big).
\end{equation}
The function 
\begin{equation}\label{hyg}
f(z)=\textstyle{\sum_{k=0}^{\infty}\frac{\textstyle{\prod_{j=1}^k}(j+x)}{(x+2)_k\cdot z^{k+1}}}
\end{equation}
 is a specific case of the generalized hypergeometric function
with the $k$-th Pochhammer symbol $(a)_k= a(a+1)\cdots(a+k-1)$ and $(a)_0 = 1$.

Instead of $\Phi_1(x,1/z)$, we employ $f(z)$ with the factor $(1+x)$, since $f(z)$ is a generalized hypergeometric function
by the expression \eqref{hyg}.
\end{definition}

To control denominators,  we define the following notations. 
\begin{notation}\label{corrected}
Let $\beta$ be an algebraic number and  $x\in\mathbb{Q}\cap[0,1)$. We put\footnote{We correct 
the notations of $\nu(x), \nu_n(x)$  in \cite[p.~650, line.~-11]{KP}) as {Notation} 3.1 above.}
\begin{align}
&{\rm{den}}(\beta)=\min\{1\leq m\in \Z \mid  \text{$m\beta$ is an algebraic integer}\},\\ 
&\nu(x)={\rm{den}}(x)\cdot\textstyle{\prod}_{\substack{q:\rm{prime} \\q|{\rm{den}}(x)}}q^{1/(q-1)},~~~~~
\nu_n(x)= {\rm{den}}(x)^n\cdot\textstyle{\prod}_{\substack{q:\rm{prime} \\ q|{\rm{den}}(x)}}q^{\lfloor n/(q-1)\rfloor},\\
&\kern-3pt d_n(x)={\rm{den}}\left(({1+x})^{-1},\ldots,({n+x})^{-1}\right) \AND { \kappa_n=\kappa_n(\beta)={\rm{den}}(x)\nu_{n}(x)d_n(x){\big(\rm{den}}(\beta)\big)^n.}
\end{align}
\end{notation}

We further prepare the notations below. 

\begin{notation}
 For a given $z\in \C$ with $|z|>1$, consider the polynomial
 \begin{equation}\label{eq: P}
    P(X) = X^2-2(2z-1)X+1.
\end{equation}
Let us fix a branch of $\sqrt{z^2-z}$. 
Denote the two roots $\lambda_1, \lambda_2$ of \eqref{eq: P} 
with  $\rho_j(z)=|\lambda_j|, (1\leq j \leq 2$ respectively) and $\rho_1(z) \leq \rho_2(z).$
Note that $ \rho_1(z)\cdot  \rho_2(z)=1$ with $|z|>1$ implies
\begin{equation}\label{strict}
\rho_1(z)< 1 <  \rho_2(z).
\end{equation}
Indeed,  if $\rho_1(z)=\rho_2(z)=1$, then
$2= \rho_1(z)+\rho_2(z) \geq 2|2z-1| \geq 2(2|z|-1) > 2$, a contradiction \cite[p.~664, line 1--5]{KP}. 
Hence, the equation (\ref{eq: P}) has two distinct simple roots.

When $\beta$ is an algebraic number with $\vert\beta\vert>1$,
for each conjugate  isomorphism  $\sigma_k~ (1\leq k \leq d)$  corresponding to the
archimedean place, consider
\begin{equation}\label{rho}
P_k(X) = X^2-2(2\sigma_k(\beta)-1)X+1~~ (1\leq k \leq d).
\end{equation}
Noting that $ \sigma_k$ is an isomorphism, we have $\sigma_k(\beta)\neq 0$ and $ \sigma_k(\beta)\neq 1$, 
hence \eqref{rho} has also two distinct simple roots. 
Denote the two roots of \eqref{rho} by $\lambda_1(k), \lambda_2(k)$ for each $1\leq k\leq d$ and put 
\begin{equation*}\label{rho-s}
 {\rho_1(\sigma_k(\beta)):=\min \{|\lambda_1(k)|, |\lambda_2(k)|\}, ~~\rho_2(\sigma_k(\beta)):=\max \{|\lambda_1(k)|, |\lambda_2(k)|\}.}
\end{equation*}
Hence it follows
 $\rho_1(\sigma_k(\beta)) \leq 1\leq \rho_2(\sigma_k(\beta))$ since $\rho_1(\sigma_k(\beta))\cdot \rho_2(\sigma_k(\beta)) =1$.
 By $|\beta|>1$ we have $\rho_1(\beta) <1<\rho_2(\beta)$, thus, 
if  we suppose that the two modulus are distinct: 
$\rho_1(\sigma_k(\beta))\neq\rho_2(\sigma_k(\beta))$ for  all $2\leq k \leq d$, then the following inequality holds: 
\begin{equation}\label{rhorho}
\rho_1(\sigma_k(\beta)) <1<\rho_2(\sigma_k(\beta)) ~(1\leq  \forall k \leq d).
\end{equation}
If  $\rho_1(\sigma_k(\beta)) =\rho_2(\sigma_k(\beta))$, by \eqref{rho} we have $\rho_1(\sigma_k(\beta)) =\rho_2(\sigma_k(\beta))=1$,
which is the excluded case, to adapt Theorem \ref{Pituk} (see {Remark} \ref{rem}).

\end{notation}

\begin{notation}\label{DQE}
Let $\beta$ be an algebraic number with $\vert\beta\vert>1$ and
$x\in\mathbb{Q}\cap[0,1)$. Define $\Delta$, $Q$, $E$ by
 \begin{align*}
    \Delta & ={\rden}(\beta)\cdot{\rexp}\left({\rden}(x)/\varphi({\rden}(x))\right)\cdot \nu(x), \\
         Q& =\Delta\cdot\textstyle{\prod}_{1\leq k \leq d}\Big( \rho_2(\sigma_k(\beta)\Big),  ~~
        E = \rho_2(\beta)/\rho_1(\beta)=\big(\rho_2(\beta)\big)^2,
         \end{align*}
where $\varphi$ denotes the Euler's totient function.
\end{notation}

\end{definition}

Our theorem is as follows.

\begin{theorem}[Main Theorem]\label{main}
Let $\beta$ be an algebraic number of degree $d$ over $\Q$ and
$x\in\mathbb{Q}\cap[0,1)$. Put $K=\Q(\beta)$. 
Suppose $\vert\beta\vert>1$ and assume $\rho_1(\sigma_k(\beta)) <\rho_2(\sigma_k(\beta))~(2\leq k \leq d)$.
Define $\Delta, Q, E$ as in {Notation} \ref{DQE}. 
If  
\begin{equation}\label{important}
{\lambda}:=\dfrac{1}{\delta}-\dfrac{\log Q}{d\,\log\,E}>0,
\end{equation}
  then $f(\beta)=(1+x) \Phi_1(x,1/\beta) \notin K$, and its effective $K$-approximation measure
satisfies
\[ \mueff(f(\beta), K)\leq {\lambda}^{-1}.\]
\end{theorem}


\section{General approximation measures}\label{lemma}
The following lemma is proven by Marcovecchio-Viola  \cite[Lemma 2.2]{marco2},
replacing $\lim_{n\to \infty}$ by 
$\limsup_{n\to \infty}$ in the assumption of \cite[Lemma 2.4]{A-V}: that is an important improvement.

\begin{lemma}\label{marco}
Let $K\subset \C$ be a number field. Let $\delta$ be the degree defined in \eqref{delta}. Let $\xi\in\C$.
Consider a sequence $\Theta_n\in K$ 
such that $\Theta_n\neq \xi$ for infinitely many $n$  and verifies 
\[\limsup_{n\to \infty}\big(\log|\xi-\Theta_n|\big)/n\leq-\rho \AND \limsup_{n\to \infty}\,\hh(\Theta_n)/n\leq c\]
for positive real numbers $\rho$ and $c$. If 
\begin{equation*}
{\lambda}:=\dfrac{1}{\delta}-\dfrac{c}{\rho}>0
\end{equation*}
then $\xi\notin K$, and  we have
$\mueff(\xi, K)\leq  {\lambda}^{-1}.$
\end{lemma}

\begin{proof}
This is  \cite[Lemma 2.2]{marco2}. The proof of the lemma is everywhere effective.
\end{proof}

\section{Pad\'e approximants satisfying a linear recurrence }\label{Pade}
Here, we recall the results in \cite[Proposition 3.5, Theorem 7]{KP} on 
the polynomials  
which form Pad\'e approximants for the function $f(z)$, which is the case $\alpha=1, \gamma=x, \delta=-x$ \cite[p.~662, line~-8, Proof of Lemma 4.2 (ii)]{KP},
satisfying 
a linear recurrence of order $2$ \cite[Lemma 3.8]{KP}.
Here we put $P_n=P_{n,1}, Q_n=P_{n,0}$ for those defined in \cite[p.~662, line -6 and -7, Proof of Lemma 4.2]{KP} also in \cite[Proposition 3.5]{KP}.

\begin{proposition}\label{pade f}
Consider the pair of polynomials $\big(Q_n(z), P_n(z)\big)_{n\geq 0}\in\Q[z]^2$ defined by
   \begin{align*}
    &Q_{n}(z)=\sum_{\ell=0}^n(-1)^{\ell}\dfrac{(n+1+x)_{n-\ell}}{(n-\ell)!}\dfrac{(n-\ell+1+x)_{\ell}}{\ell!}z^{n-\ell},\\
    &P_{n}(z)=\sum_{\ell=0}^{n-1}\left(\sum_{k=\ell}^{n-1}(-1)^{n-k-1}\dfrac{(n+1+x)_{k+1}}{(k+1)!}\dfrac{(k+2+x)_{n-k-1}}{(n-k-1)!}\dfrac{(1+x)}{k-\ell+1+x}\right)z^{\ell}.
    \end{align*}
    
Then we have the following properties \rm{(i)}\rm{(ii)}\rm{(iii)}\rm{(iv)}.\\
({\textit{i}}) For each $n$, the pair $\big(Q_n(z), P_n(z)\big)_{n\geq 0}$ forms Pad\'{e} approximants of $f(z)$ of weight $n$. By
 \cite[Proposition 3.5]{KP},
 the remainder function 
    $R_n(z)= Q_n(z)f(z)-P_n(z)$ satisfies $\lim_{n\to \infty}|R_n(z)|^{1/n}=0$ when  $|z|>1$, where
        \[
        R_n(z)=\sum_{\ell=n}^{\infty}\binom{\ell}{n}\dfrac{\prod_{i=1}^\ell(i+x)\cdot n!}{(x+2)_{n+\ell}}\dfrac{1}{z^{\ell+1}}\in\Big(\dfrac{1}{z^{n+1}}\Big),
    \]
    with  $\Big(1/z^{n+1}\Big)$ an ideal in the power series domain $\Q\left[\left[1/z\right]\right]$.

\smallskip

\noindent
({\textit{ii}})
By  \cite[p.~658, line 2, Lemma 3.7]{KP}, we have
  \[\det \begin{pmatrix}Q_n(z) & P_n(z) \\ Q_{n+1}(z) & P_{n+1}(z)\end{pmatrix}\neq 0.\]
\\
({\textit{iii}})
All of $P_n(z), Q_n(z)$ and $ R_n(z)$ satisfy the linear recurrence  \cite[Lemma 3.8 and Theorem 7.1]{KP}:
       \begin{align}\label{recurrence}
    A_nX_{n+1} -(z-B_n)X_n + C_nX_{n-1} = 0 \qquad (n\geq 1)
    \end{align}
    \[\text{with}~
        A_n=\dfrac{(n+x+1)(n+1)}{(2n+x+1)(2n+x+2)} , \ \ B_n=\dfrac{2n^2+(1+x)(2n+x)}{(2n+x)(2n+x+2)}, \ \
        C_n=\dfrac{n(n+x)}{(2n+x)(2n+x+1)},
    \]
for which we have $A_n\to 1/4, B_n\to 1/2$ and $ C_n\to 1/4~~(n\to \infty)$.
Hence the characteristic equation of the linear recurrence exists, which is of form
 \begin{equation}\label{charpolref}
 P(X) = X^2-2(2z-1)X+1.
 \end{equation}

\noindent
({\textit{iv}})
Suppose that 
the two roots of \eqref{charpolref} are distinct in modulus. Then by \cite[Theorem 7.1 (32)]{KP}, the sequence 
$X_n$ in \eqref{recurrence} satisfies,  for one modulus $\rho$ of the roots of \eqref{charpolref}:
\[\vert X_n \vert \leq \big(C\cdot \rho^n\big)/\sqrt{n}\]
with an effective constant $C>0$ that can be explicitly expressed depending only on $f(\beta)$.
In particular, the sequence  $X_n$ 
converges and we have $\lim_{n\to \infty}  \big(\log|X_n|\big)/n =\log \rho$.
\end{proposition}

\begin{proof}
The statement of this proposition directly follows from \cite{KP} as mentioned above.
\end{proof}

\begin{corollary}\label{det}
The sequences
$Q_n(z)$ and $R_n(z)$ form linearly independent solutions of the linear recurrence \eqref{recurrence} $($in other words,  
 independent solutions of the difference equation$)$.
\end{corollary}

\begin{proof}
Since $R_n(z)= Q_n(z)f(z)-P_n(z)$ and $f(z)$ not identically zero,   by Proposition \ref{pade f} (ii),
\[\det \begin{pmatrix}Q_n(z) & R_n(z) \\ Q_{n+1}(z) & R_{n+1}(z)\end{pmatrix}\neq 0\text{~~follows~.}\]
\end{proof}

\section{The Poincar\'e-Perron theorem}\label{PP}
In this section, we present Perron's Second theorem and a theorem of
Pituk \cite[Theorem 8.47]{E} \cite[Theorem 2]{Pituk}  giving asymptotic estimates of the solutions
of a linear difference equation.  

Refer to \cite{E} for basic properties of
the difference equation with a historical survey. 
Perron's Second theorem is indeed already
valid in the case where the roots of the characteristic polynomial of the linear recurrence
are not necessarily distinct;  it is explicitly explained in  \cite[Theorem C (line 14, p.\,203)]{Pituk} 
\cite[Theorem\,8.11, p.\,344]{E}.
Moreover, it is proven as theorems \cite[Theorem 1, 2]{Pituk}  \cite[Theorem\,8.47, p.\,389]{E} 
for any non-trivial solution of the
difference equation \cite[line\,1--3, p.\,388]{E}, under the explicit hypothesis 
assuming that there should be at least two roots whose moduli are distinct
\cite[line 3, p. 208]{Pituk}.
Such generalization might enable us to give precise estimates in Diophantine problems.
of moduli partly equal. 

Although our linear recurrence equation of order $2$ has the characteristic polynomial 
\eqref{rho} with distinct roots,
we adapt the following generalization due to Pituk \cite[Theorem 2]{Pituk},
to ensure the effectivity of our $K$-approximation measures, namely, we employ Theorem \ref{Pituk}.

\begin{definition}\label{assum}
Let $\ell$ be a positive integer and let $s_{j}(n) ~(1\leq  j \leq \ell)$ be a function from $n\in \mathbb{Z}_{\geq 0}$
to $\C$. 
Consider the $\ell$-th order linear difference equation
in unknown complex-valued functions $x(n)$ on $n\in \mathbb{Z}_{\geq 0}$ satisfying
		\begin{equation}\label{deq2}
x(n+\ell)+s_{1}(n)x(n+\ell-1)+\cdots+s_{\ell}(n)x(n)=0.
		\end{equation}
		Suppose that the limit $t_j:=\lim_{n\to \infty} s_{j}(n)$ exists in $\C$ for each $1\leq j \leq \ell$, and
		the order $\ell$	is the minimal number where such recurrence \eqref{deq2} holds.
Let us write the characteristic equation
\begin{equation}\label{deq1}
		\lambda^\ell+t_1\lambda^{\ell-1}+\cdots+ t_\ell=0, 
				\end{equation}	
and denote by $\lambda_1, \ldots,  \lambda_\ell$ the roots of the equation \eqref{deq1} and by 
$\rho_1, \ldots,  \rho_\ell$ their modulus respectively.
\end{definition}

First we present the following theorem \cite[Theorem C, p.~203]{Pituk}.
\begin{theorem}[Perron's Second theorem]\label{perron2}
Consider the difference equation \eqref{deq2} defined in {Definition} \ref{assum}. Suppose $s_{\ell}(n)\neq 0$ for all $n$. 
Then \eqref{deq2} has $\ell$ linearly independent solutions $x_1(n), \ldots, x_{\ell}(n)$,
such that for each $1\leq j \leq \ell$,  
\begin{equation}\label{perron}
\displaystyle \limsup_{n\rightarrow \infty} \big(\log |x_j(n)|\big)/n=\log\vert\lambda_j\vert,
\end{equation}
where $\lambda_j$ is one root of
 \eqref{deq1}, with $1\leq j \leq \ell$. 
\end{theorem}
In {Theorem} \ref{perron2}, 
the roots $\lambda_1, \ldots,  \lambda_\ell$ are not necessarily supposed to be distinct  \cite[line\,14, p.\,344]{E}, 
however, by  the proof of Pituk \cite[line\,3, p.\,208]{Pituk},  there should be at least two distinct moduli 
among those of the roots of the characteristic polynomial. 

The following statement is a theorem of  Pituk \cite[Theorem 2]{Pituk}, 
whose proof  is everywhere {\textit{effective}}, that we apply in this article to deduce effective $K$-approximation measures.

\begin{theorem}[Pituk]\label{Pituk}
Consider the difference equation \eqref{deq2} defined in {Definition} \ref{assum}.
Then for any solution $x(n)$ of  \eqref{deq2}, we have
either $x(n)=0$ for all large $n$, or the following property holds. Put 
\begin{equation}\label{pm}
\log\varrho:=\limsup_{n\rightarrow \infty} \big(\log {|x(n)|\big)/n}.
\end{equation}
Then  $\log\varrho$ is equal to $\log\vert\lambda_j\vert$ where $\lambda_j$ is one root of
 \eqref{deq1}. 
\end{theorem}

\medskip

Consequently, {\textit{every}} linearly independent solution corresponds to one root of  \eqref{deq1}, and any 
non-zero solution of \eqref{deq2} has the {asymptotic} behavior above. 
Note that this limit sup could not be replaced by the limit in Theorem \ref{Pituk} \cite[p.\,205]{Pituk}, but in our case, the limits exist.

\begin{remark}\label{rem}
As is mentioned in Section~\ref{Intro}, Pituk gave an effective proof 
for Theorem \ref{Pituk}, supposing that at least two moduli are distinct,
among those of the roots of \eqref{deq1}.
We have supposed 
$\rho_1(\sigma_k(\beta)) < \rho_2(\sigma_k(\beta))~~(2\leq k\leq d)$ and $|\beta|>1$ (which gives  $\rho_1(\beta) < \rho_2(\beta)$\,) 
in {Theorem} \ref{main}, hence we are in the case
(our Proposition \ref{pade f} (iv) holds when the roots of the
characteristic equation are distinct in modulus and the limits exist).

\end{remark}

\section{Proof of main theorem}\label{proof}
Now let us consider the characteristic polynomial \eqref{rho}, defined by
$P_k(X) = X^2-2(2\sigma_k(\beta)-1)X+1$ for each conjugate $\sigma_k~(1\leq k \leq d)$.
We rely on the fact that \eqref{pm} in Theorem \ref{Pituk} 
yields {asymptotic} estimates. Since $\sigma_k~(1\leq k \leq d)$ is an isomorphism keeping stable the rationals, 
Proposition \ref{pade f} ensures that $P_{n}(\sigma_k(\beta))$, $Q_{n}(\sigma_k(\beta))$ for each $1\leq k \leq d$, then $R_{n}(\beta)$ satisfy
the same linear recurrence equation. 
\begin{lemma} \label{apply P-P}
    Let $\beta$ be an algebraic number
    with $|\beta|>1$. Suppose $\rho_1(\sigma_k(\beta)) <\rho_2(\sigma_k(\beta))$ for all ~$2\leq k \leq d$. 
Then, 
as $n$ tends to the infinity, we obtain
    \begin{equation}\label{eq: estimates for P_n and Q_n}
      \lim_{n\to \infty}  |Q_{n}(\sigma_k(\beta))^{1/n}|= \rho_2(\sigma_k(\beta))
\AND
     \lim_{n\to \infty} |R_n(\beta)^{1/n}|=\rho_1(\beta).
         \end{equation}
\end{lemma}

\begin{proof} 
{Corollary} \ref{det} and {Theorem} \ref{Pituk} giving a generalization of {Theorem} \ref{perron2} imply, either
    \begin{align}\label{numero1}
      \limsup_{n\to \infty}  |Q_{n}(\sigma_k(\beta))^{1/n}|= \rho_2(\sigma_k(\beta))
\AND
     \limsup_{n\to \infty} |R_n(\beta)^{1/n}|=\rho_1(\beta)
    \end{align}
    or
    \begin{align}\label{numero2}
      \limsup_{n\to \infty}  |Q_{n}(\sigma_k(\beta))^{1/n}|= \rho_1(\sigma_k(\beta))
\AND
     \limsup_{n\to \infty} |R_n(\beta)^{1/n}|=\rho_2(\beta),
    \end{align}
  since $Q_n(z)$ and $R_{n}(z)$ form linearly independent solutions of \eqref{charpolref}. 

Now, taking into account of  \cite[Lemma 4.3]{KP} together with 
$\lim_{n\to \infty}|R_n(\beta)|^{1/n}=0$ from Proposition \ref{pade f} (i) 
and $\rho_1(\sigma_k(\beta)) <1< \rho_2(\sigma_k(\beta))$ by \eqref{rhorho}, 
Theorem \ref{Pituk} guarantees that  \eqref{numero1} holds instead of \eqref{numero2}, since
each linearly independent solution of the linear recurrence \eqref{recurrence} corresponds to one root of
the characteristic equation. 
Precisely, we have (noting that $o(1)$ can be explicit)
\begin{equation}\label{numero3}
\Big(\log |Q_{n}(\sigma_k(\beta))|\Big)/n= (1+o(1))\Big(\log\rho_2(\sigma_k(\beta))\Big)~~(n\to \infty),
  \end{equation}
  \begin{equation*}
\Big(\log|R_n(\beta)|\Big)/n= (1+o(1))\Big(\log\rho_1(\beta)\Big)~~(n\to \infty).
    \end{equation*}
  \end{proof}
\begin{corollary}\label{corcor}
Let $\beta\in K$ and put $\Theta_n={P_{n}(\beta)}/{Q_{n}(\beta)}.$ 
Then we have 
\\
$({\rm{i}})$\,~$\lim_{n\to \infty}h\big(\Theta_n\big)/n\leq\Big(\log Q\Big)/d$,
\\
$({\rm{ii}})$~$\lim_{n\to \infty}\Big(\log|\xi-\Theta_n|\Big)/n=-\log E$ $($in particular $\xi\neq \Theta_n$ for infinitely many $n$$)$.

\end{corollary}

\begin{proof}
By \eqref{eq: estimates for P_n and Q_n}, we obtain $\lim_{n\to \infty}\mathrm{h_{\infty}}(\Theta_n)/n=
\big(1/d\big)\sum_{1\leq k \leq d}\log \rho_2(\sigma_k(\beta))$, noting that 
Proposition \ref{pade f} (iv) impies also for $P_n$ the estimate
$\Big(\log |P_{n}(\sigma_k(\beta))|\Big)/n= (1+o(1))\Big(\log\rho_2(\sigma_k(\beta))\Big)$,
since $P_{n}(\sigma_k(\beta))$  satisfies the same recurrence as $Q_{n}(\sigma_k(\beta))$.
Combining  \cite[Lemma 4.1, Lemma 4.2(ii)]{KP}  with the definition of
 $\Delta,~Q$ in {Notation} \ref{DQE}, we obtain that  $\Delta^n$ is the denominator of 
 $P_{n}(\sigma_k(\beta))$ and $Q_{n}(\sigma_k(\beta))$. Then we have ${h}_{f}\big(\Theta_n\big)/n\leq
 \big(1/d\big)\log \Delta$, hence
 $({\rm{i}})$ follows.  
We are going to prove  $({\rm{ii}})$. 
By $R_n(\beta)=Q_n(\beta)f(\beta)-P_n(\beta)$,  we obtain for $\xi=f(\beta)$: 
$$\limsup_{n\to \infty}\Big(\log|\xi-\Theta_n|\Big)/n=\limsup_{n\to \infty}\Big(\log|Q_{n}(\beta)\xi-P_{n}(\beta)|\Big)/n
 -\liminf_{n\to \infty}\Big(\log|Q_{n}(\beta)|\Big)/n.$$
Noting that the both limits do exist on the right-hand side by  {Lemma} \ref{apply P-P} with \eqref{numero3} and $k=1$:
$$\lim_{n\to \infty}\Big(\log|Q_{n}(\beta)\xi-P_{n}(\beta)|\Big)/n=\log\rho_1(\beta)\AND
\lim_{n\to \infty}\Big(\log|Q_{n}(\beta)|\Big)/n=\log\rho_2(\beta), \text{~we conclude}$$
$$\lim_{n\to \infty}\Big(\log|\xi-\Theta_n|\Big)/n=\log\Big(\rho_1(\beta)/\rho_2(\beta)\Big).$$
Consequently,  by definition of 
$E = \rho_2(\beta)/\rho_1(\beta)=\big(\rho_2(\beta)\big)^2$, we obtain  $({\rm{ii}})$.
\end{proof}

The proof of Theorem \ref{main} is carried out as follows. We employ the
Pad\'e approximants $(P_{n}(\beta))_{n\ge 0}$, $(Q_{n}(\beta))_{n\ge 0}$ for $\xi=f(\beta)$.
We choose  $\rho=\log E,  ~c=\Big(\log Q\Big)/d$,  then by 
{Corollary} \ref{corcor} together with {Lemma} \ref{marco}, we obtain the statement.

\vspace{-2mm}
\section{Examples}\label{example}

We list a few of new examples of $K$-approximation measures satisfying the assumption \eqref{important}.
Denote by $\alpha$, a root with  $\Re(\alpha )>0$, $\Im(\alpha )>0$ of $2X^4-3X^2+2$ considered in \cite{AD}. The
number  $3+2\alpha$ below is an algebraic integer. Denote by $\omega\neq 1$ the cubic root of the unity with $\Im(\omega)>0$.

\noindent
\begin{tabular}{>{\centering}p{7em}|>{\centering}p{8em}||>{\centering}p{3em}|>{\centering}p{5em}|p{7.7em}|p{8em}}
$\beta$ & $f(\beta )$ & $x$ & $K$ & $\mueff(f(\beta ),K)\leq~~~$ & results in \cite{A-R} or \cite{A-V}\\
\hline\hline
$2$ & $\log 2$ & $0$ & $\Q$ & 4.6221 &  4.6221~~\cite{A-R}\\
\end{tabular}\\
\begin{tabular}{>{\centering}p{7em}|>{\centering}p{8em}||>{\centering}p{3em}|>{\centering}p{5em}|p{7.7em}|p{8em}}
\hline
$2i$ & $\log \frac{4-2i}{5}$ & $0$ & $\Q (i)$ & 2.61631&  -  \\
$2i$ & & $1/3$ & $\Q (i)$ & 7.73819&  -  \\
$2i$ & & $1/5$ & $\Q (i)$ & 8.63437&  -  \\
\end{tabular}\\
\begin{tabular}{>{\centering}p{7em}|>{\centering}p{8em}||>{\centering}p{3em}|>{\centering}p{5em}|p{7.7em}|p{8em}}
\hline
$2^{24}$ & $\log \frac{2^{24}}{2^{24}-1}$ & $0$ & $\Q$ & 2.1175&  -  \\
$2^{24}$ & & $1/3$ & $\Q$ & 2.42328&  -  \\
$2^{24}$ & & $1/5$ & $\Q$ & 2.44198&  -  \\
\hline
$5^{10}+5^{10}i$ & $\log \frac{2\cdot 5^{20}-5^{10}-5^{10}i}{2\cdot 5^{20}-2\cdot 5^{10}+1}$ & $0$ & $\Q (i)$ & 2.06402\\
$5^{10}+5^{10}i$ & & $1/3$ & $\Q (i)$ & 2.20127&  -  \\
$5^{10}+5^{10}i$ & & $1/5$ & $\Q (i)$ & 2.20906&  -  \\
\hline
\end{tabular}\\
\begin{tabular}{>{\centering}p{7em}|>{\centering}p{8em}||>{\centering}p{3em}|>{\centering}p{5em}|p{7.7em}|p{8em}}
$1000+\sqrt[3]{5}\omega$ & $\log \frac{1000+\sqrt[3]{5}\omega }{999+\sqrt[3]{5}\omega }$ & $0$ & $\Q (\sqrt[3]{5}\omega )$ & 6.82514&  -  \\
$1000+\sqrt[3]{5}\omega$ & & $1/3$ & $\Q (\sqrt[3]{5}\omega )$ & 9.67602&  -  \\
$1000+\sqrt[3]{5}\omega$ & & $1/5$ & $\Q (\sqrt[3]{5}\omega )$ & 9.89516&  -  \\
\end{tabular}\\
\begin{tabular}{>{\centering}p{7em}|>{\centering}p{8em}||>{\centering}p{3em}|>{\centering}p{5em}|p{7.7em}|p{8em}}
\hline
$3+2\alpha $ & $\log \frac{3+2\alpha }{2+2\alpha }$& $0$ & $\Q (\alpha )$ & 11.2027&  -  \\
\end{tabular}\\
\begin{tabular}{>{\centering}p{7em}|>{\centering}p{8em}||>{\centering}p{3em}|>{\centering}p{5em}|p{7.7em}|p{8em}}
\hline
$50+50\sqrt[4]{3}i$ & $\log \frac{50+50\sqrt[4]{3}i}{49+50\sqrt[4]{3}i}$ & $0$ & $\Q (\sqrt[4]{3}i)$ & 163.837&  -  \\
\end{tabular}

\noindent
\begin{tabular}{>{\centering}p{7em}|>{\centering}p{8em}||>{\centering}p{3em}|>{\centering}p{5em}|p{7.7em}|p{8em}}
\hline
$-6-5\sqrt{2}$ & $\log (8-5\sqrt{2})$ & $0$ & $\Q (\sqrt{2})$ & 6.47612 & 10.8912~~\cite{A-V} \\
$-6-5\sqrt{2}$ & & $1/3$ & $\Q (\sqrt{2})$ & 50.0916 & -\\
$-6-5\sqrt{2}$ & & $1/5$ & $\Q (\sqrt{2})$ & 77.9114 & -\\
\hline
$18+12\sqrt{2}$ & $\log (18-12\sqrt{2})$ &  $0$ & $\Q (\sqrt{2})$ & 5.49683 & 7.6234~~\cite{A-V}\\
$18+12\sqrt{2}$ & &  $1/3$ & $\Q (\sqrt{2})$ & 13.7134 & -\\
$18+12\sqrt{2}$ & &  $1/5$ & $\Q (\sqrt{2})$ & 14.8937 & -\\
\end{tabular}\\
\begin{tabular}{>{\centering}p{7em}|>{\centering}p{8em}||>{\centering}p{3em}|>{\centering}p{5em}|p{7.7em}|p{8em}}
\hline
$3364+2378\sqrt{2}$ & $\log (3364-2378\sqrt{2})$ &  $0$ & $\Q (\sqrt{2})$ & 4.44656 & 4.9905~~\cite{A-V}\\
$3364+2378\sqrt{2}$ & &  $1/3$ & $\Q (\sqrt{2})$ & 5.80557 & -\\
$3364+2378\sqrt{2}$ & &  $1/5$ & $\Q (\sqrt{2})$ & 5.9012 & -\\
\end{tabular}\\
\begin{tabular}{>{\centering}p{7em}|>{\centering}p{8em}||>{\centering}p{3em}|>{\centering}p{5em}|p{7.7em}|p{8em}}
\hline
$\dfrac{\sqrt[3]{7/6}}{\sqrt[3]{7/6}-1}$ & $\dfrac{1}{3}\log (7/6)$ &  $0$ & $\Q \left( \sqrt[3]{7/6}\right)$ & 11.5787 & 43.9427~~\cite{A-V}\\
$\dfrac{\sqrt[3]{7/6}}{\sqrt[3]{7/6}-1}$ & &  $1/3$ & $\Q \left( \sqrt[3]{7/6}\right)$ & 240.384 & -\\
\hline
$\dfrac{\sqrt[3]{2}}{\sqrt[3]{2}-1}$ & $\dfrac{1}{3}\log 2$ &  $0$ & $\Q \left( \sqrt[3]{2}\right)$ & 22.4389 & -\\
\hline
\end{tabular}\\


\vspace{3mm}

\noindent
\textbf{Acknowledgement}
\\
The authors are indebted to the anonymous referee for her/his helpful comments
and suggestions on an earlier version of the manuscript. 
The first author is supported by JSPS KAKENHI Grant no.~21K03171 and
also partly supported by the Research Institute for Mathematical Sciences, an international joint usage
and research center located in Kyoto University.
\vspace{-3mm}

\bibliography{}

\

\begin{minipage}[t]{0.39\textwidth}
Noriko Hirata-Kohno, \\
Department of Mathematics,\\
College of Science and Technology,\\
Nihon University, Kanda, Chiyoda, \\
Tokyo, 101-8308, Japan\\
hiratakohno.noriko@nihon-u.ac.jp
\end{minipage}
\begin{minipage}[t]{0.3\textwidth}
Ryuji Muroi,\\
Chiba Nihon University\\
First Senior and Junior \\High School
\end{minipage}
\begin{minipage}[t]{0.3\textwidth}
Yusuke Washio,\\Mathematics Major, \\
Graduate School \\of Science and Technology,\\
Nihon University
\end{minipage}

\end{document}